\documentclass[11pt]{amsart}
\usepackage{amsmath,amssymb,amsthm,amscd,verbatim}
\usepackage[T1]{fontenc}
\usepackage{graphicx}
\usepackage{epsfig}
\usepackage{hyperref}
\usepackage[dvipsnames]{xcolor}

\hypersetup{
  colorlinks   = true,
  linkcolor = red!70!black,
     citecolor    = green!50!black
}

\newtheorem{thm}{Theorem}[section]

\newtheorem{cor}[thm]{Corollary}

\newtheorem{cla}[thm]{Claim}
\newtheorem*{thm*}{Theorem}

\theoremstyle{definition}

\theoremstyle{remark}

\newtheorem*{acknowledgement}{Acknowledgments}

\newcommand{\mdc}{\ensuremath{\mathrm{maxdicut}}}
\newcommand{\twr}{\ensuremath{\mathrm{twr}}}

\newcommand{\opt}{\ensuremath{\mathrm{OPT}}}
\newcommand{\cut}{\ensuremath{\mathrm{CUT}}}

\renewcommand{\Pr}{\,\mathbb{P}}

\renewcommand{\le}{\leqslant}
\renewcommand{\leq}{\leqslant}
\renewcommand{\ge}{\geqslant}
\renewcommand{\geq}{\geqslant}

\def\longequation{$$\vcenter\bgroup\advance\hsize by -9em%
\noindent\ignorespaces\refstepcounter{equation}}%
\makeatletter%
\def\endlongequation{\egroup\eqno(\theequation)$$\global\@ignoretrue}
\makeatother

\begin{document}
\title[Local approximation of the Maximum Cut in regular graphs]{Local approximation of the Maximum Cut\\ in regular graphs}

\author{\'Etienne Bamas} \address{School of Computer and Communication
  Sciences, \'Ecole Polytechnique F\'ed\'erale de Lausanne,
  Switzerland}
\email{etienne.bamas@epfl.ch}

\author{Louis Esperet} \address{Laboratoire G-SCOP (CNRS, Univ. Grenoble Alpes), Grenoble, France}
\email{louis.esperet@grenoble-inp.fr}

\thanks{An extended abstract of this work  appeared in the proceedings
  of the 45th International Workshop on Graph-Theoretic Concepts in Computer Science (WG), 2019.\newline Partially supported by ANR Projects GATO (\textsc{anr-16-ce40-0009-01}) and GrR (\textsc{anr-18-ce40-0032}), and LabEx PERSYVAL-Lab
  (\textsc{anr-11-labx-0025}).}

\date{}
\sloppy

\begin{abstract} 
  This paper is devoted to the distributed complexity of finding an approximation of the maximum cut ({\sc MaxCut}) in graphs. A classical algorithm consists in letting each vertex choose its side of the cut uniformly at random. This does not require any communication and achieves an approximation ratio of at least $\tfrac12$ in expectation. When the graph is $d$-regular and triangle-free, a slightly better approximation ratio can be achieved with a randomized algorithm running in a single round. Here, we investigate the round complexity of \emph{deterministic} distributed algorithms for  {\sc MaxCut} in regular graphs. We first prove that if $G$ is $d$-regular, with $d$ even and fixed, no deterministic algorithm running in a constant number of rounds can achieve a constant approximation ratio. We then give a simple one-round deterministic algorithm achieving an approximation ratio of $\tfrac1{d}$ for $d$-regular graphs when $d$ is odd. We show that this is best possible in several ways, and in particular no deterministic algorithm with approximation ratio $\tfrac1{d}+\epsilon$ (with $\epsilon>0$) can run in a constant number of rounds. We also prove results of a similar flavour for the  {\sc MaxDiCut} problem in regular oriented graphs, where we want to maximize the number of arcs oriented from the left part to the right part of the cut.

\end{abstract}
\maketitle

\section{Introduction}

Although the maximum cut problem ({\sc MaxCut}) is fundamental in combinatorial optimization, it has not been intensively studied from the perspective of distributed algorithms. The folklore algorithm consisting in choosing uniformly at random one side of the cut for each vertex of a graph $G$ can however be seen as a distributed randomized algorithm with no rounds of communication. By the linearity of expectation, this algorithm gives a cut (a bipartition of the vertex set) of size at least $m/2$ in expectation, where $m$ is the number of edges of $G$. Here, by the \emph{size} of the cut, we mean the number of edges connecting the two parts of the bipartiton. Since every cut in $G$ contains at most $m$ edges, this algorithm has \emph{approximation ratio} at least $\tfrac12$ in expectation, which means that the size of the cut given by the algorithm is at least $\tfrac12$ of the size of the maximum cut in expectation.

A natural question is whether a better approximation ratio can be obtained if more rounds of communications are allowed. This question was answered positively by Shearer \cite{S92} in the case of triangle-free $d$-regular graphs. A \emph{$d$-regular} graph is a graph in which every vertex has degree $d$. In the case of triangle-free $d$-regular graphs, Shearer gave a simple randomized algorithm finding a cut of size at least $ m\cdot (\frac{1}{2}+\frac{0.177}{\sqrt{d}})$ in expectation, and thus achieving an approximation ratio of $(\frac{1}{2}+\frac{0.177}{\sqrt{d}})$ in expectation. Shearer's algorithm
uses a single round of communication, messages consisting of a single bit, and at most 3 random bits per vertex. This was recently improved by Hirvonen, Rybicki, Schmid and Suomela \cite{H17}, who obtained a simpler algorithm finding a cut of size at least $m\cdot\left(\frac{1}{2}+\frac{0.28125}{\sqrt{d}} \right)$ in expectation. Their algorithm
uses a single round of communication, messages consisting of a single bit, and a single random bit per vertex.


The case where $d$ is small and the \emph{girth} (length of a shortest cycle) is large has also been considered: for 3-regular graphs, Kardo\v{s}, Kr\'al' and Volec \cite{K12} showed that when the girth is at least 637789, there exists a randomized distributed algorithm that outputs a cut of average size at least $0.88672m$ in at most 318894 rounds (the important value here is the size of the cut). This was improved by Lyons~\cite{L17}, who proved a lower bound of $0.89m$ for cubic graphs of girth at least 655. The best known lower bound for cubic graphs of large girth, $0.90m$, was proved by Gamarnik and Li~\cite{GL18}, using a result of Cs\'oka, Gerencs\'er, Harangi, and Vir\'ag~\cite{CGHV15}. The bound of Lyons~\cite{L17} holds for any $d$-regular graphs of large enough (but constant) girth: such graphs have a cut of size at least $ m\cdot (\frac{1}{2}+\frac{2}{\pi\sqrt{d}})\approx m\cdot (\frac{1}{2}+\frac{0.637}{\sqrt{d}})$. On the other hand, Dembo, Montanari and Sen~\cite{DMS17} showed that in random $d$-regular graphs, the maximum cut has size $m\cdot (\frac{1}{2}+\frac{0.763+o(1)}{\sqrt{d}})+o(m)$ with high probability, proving a conjecture of~\cite{GL18}. The existence of this constant $\approx 0.763$ is also connected to a conjecture of Hatami, Lov\'asz and Szegedy~\cite{HLS14} on limits of sparse graphs (see also the conclusion of~\cite{RV17} where the conjecture is strongly disproved for maximum independent sets, improving on an earlier result of~\cite{GS14}).

\medskip

All the results mentioned above (except the result of Gamarnik and Li~\cite{GL18})\footnote{Their algorithm has two phases: The first consists in finding a large bipartite induced subgraph (this can be done in a constant number of rounds), and the second phase greedily assigns each remaining vertex to the side of the bipartition where it has fewer neighbors (it is unlikely that this greedy procedure can be performed in a constant number of rounds, and at least Theorem~\ref{thm:lb} shows that it cannot be performed in a constant number of rounds when there are no restrictions on the girth).} can be translated into algorithms working in the \textsf{CONGEST} model in a constant number of rounds. In this model, each node of the graph corresponds to a processor with infinite computational power and has a unique ID (each ID is an integer between 1 and $\text{poly}(n)$, where $n$ denotes the number of vertices in the graph). Nodes can communicate with their neighbors in the graph in synchronous rounds until each node outputs 0 or 1, corresponding to its side in the cut. In the \textsf{CONGEST} model, each message sent by a node to a neighbor has size $O(\log n)$, while in some of the algorithms above, the messages have size at most 1. Let us call $\textsf{CONGEST}(B)$ the variant of the \textsf{CONGEST} model in which messages are restricted to have size at most $B$ (instead of  $O(\log n)$), and let us say that an algorithm is \emph{local} in a model if it runs in a constant number of rounds in this model. In particular the results of~\cite{H17,K12, S92} mentioned above can be translated into local algorithms in the \textsf{CONGEST($O(1)$)} model, while the results of~\cite{CGHV15,L17} can be translated into local algorithms in the \textsf{CONGEST} model.

Note that some of our lower bounds are also valid in the less restricted \textsf{LOCAL} model where the size of each message is not limited. In the following, we will make it clear if this applies. On the other hand, all our algorithms can be implemented in the \textsf{PO} model (anonymous network with port numbering and orientations), whose assumptions are significantly weaker than the \textsf{CONGEST} model (see~\cite{GHS13} for some results on local algorithms in \textsf{PO} and \textsf{CONGEST}).

\medskip

We now review recent results on distributed approximation of {\sc MaxCut}. On the deterministic side, Censor-Hillel, Levy, and Shachnai \cite{K17} designed a deterministic  $\frac{1}{2}$-approximation that runs in $\tilde{O}\left(\Delta+\log^* n \right)$ rounds in the \textsf{CONGEST} model  on any graph of maximum degree at most $\Delta$. More recently, Kawarabayashi and Schwartzman \cite{S17} improved the complexity for constant factor approximation by providing a deterministic $\left(\frac{1}{2}-\epsilon\right)$-approximation that runs in $O(\log^* n)$ rounds (for any $\epsilon>0$), in the \textsf{CONGEST} model. However, no deterministic \emph{local} approximation for {\sc MaxCut} (i.e.\ running in a constant number of rounds) in the \textsf{CONGEST} model is known.

\medskip

There is a similar gap between randomized and deterministic approximations for the maximum directed cut problem. Censor-Hillel, Levy, and Shachnai \cite{K17} provided a deterministic algorithm running in $O(\Delta+\log^* n)$ rounds that guarantees a $\frac{1}{3}$-approximation as well as a randomized $\frac{1}{2}$-approximation with the same round complexity. The round complexities were improved by Kawarabayashi and Schwartzman \cite{S17} who provided a deterministic $\left(\frac{1}{3}-\epsilon\right)$-approximation running in $O(\log^* n)$ rounds as well as a randomized $\left(\frac{1}{2}-\epsilon \right)$-approximation in $O(\epsilon^{-1})$ rounds.  All these results are stated in the \textsf{CONGEST} model. Similarly, no deterministic local algorithm is known to achieve a constant factor approximation for this problem.

\subsection{Our results} 

Our work focuses on bridging the gap between extremely efficient randomized local algorithms and slower deterministic algorithms for {\sc MaxCut}. It should be noted that there are generic tools to derandomize distributed algorithms (see~\cite{CKP16, GHK17} for recent results in this direction) but existing techniques mainly apply to \emph{locally checkable} problems (problem for which a solution can be checked locally), which is not the case for (approximations of) {\sc MaxCut}.

In Section \ref{sec_up} we show that any  deterministic algorithm that guarantees a constant factor approximation for {\sc MaxCut} on the class of bipartite \textit{d-regular} graphs when $d$ is a (constant) even integer requires $\Omega(\log^*n)$ rounds, which matches the complexity of the algorithm of Kawarabayashi and Schwartzman \cite{S17} mentioned above. When $d$ is odd, we show that one cannot achieve an approximation ratio better than $\frac{1}{d}$ in a constant number of rounds. Our proofs use an elementary graph construction and then apply Ramsey's theorem \cite{F30}. Both of these arguments are not new in distributed algorithms: our construction is inspired by Linial's seminal paper \cite{L92} that provides a lower bound on the round complexity of coloring cycles and by a more recent paper by  {\AA}strand, Polishchuk, Rybicki, Suomela, and Uitto \cite{A10} which applies Ramsey's theorem in a similar setting to prove that there is no deterministic and local constant factor approximation for the maximum matching problem. Note that similar arguments were also used by Czygrinow, Hanckowiak, and Wawrzyniak~\cite{CHW08} to prove lower bounds for the approximation of maximum independent sets in cycles. Our results hold for any $d$-regular graph ($d$ is not necessarily equal to 2), so some additional work needs to be done compared to the simple case of cycles.

In Section \ref{sec:low}, we show that this barrier of $\frac{1}{d}$ when $d$ is odd is sharp: we first remark that a result of Naor and Stockmeyer \cite{N93} on weak 2-coloring of graphs directly gives a deterministic local algorithm that guarantees a $\frac{1}{d}$-approximation. We then provide a much simpler and faster deterministic local algorithm achieving the same approximation ratio. It runs in a single round with messages of size $O(\log n)$ and we show that this cannot be improved.

For the Maximum Directed Cut problem in $d$-regular graphs, we prove that a similar situation occurs. If $d$ is even, a constant factor approximation cannot be achieved in $o(\log^* n)$ rounds, and if $d$ is odd, no $(\tfrac2{d}+\epsilon)$-approximation can be achieved in $o(\log^* n)$ rounds (for any $\epsilon>0$). On the other hand, if $d$ is odd, a $(\tfrac{2}{d+1/d})$-factor approximation can be achieved in 0 rounds, and a $(\tfrac{2}{d+1/d-3/d^2+O(d^{-3})})$-factor approximation can be achieved in 2 rounds. Note that there is a small gap between the lower bounds and the upper bound of $\tfrac{2}{d}$, and we explain some obstacles towards closing the gap.

\medskip

Our results imply that while finding a constant factor approximation for the (directed) maximum cut problem in regular graphs of even degree does not require any communication for randomized distributed algorithm (i.e.\ it can be solved in 0 rounds), for deterministic algorithms  an unbounded number of rounds is needed in this case. Note that this separation is not possible for locally checkable problems (see Theorem 3 in~\cite{CKP16}). The (perhaps) surprising aspect is that in the case of regular graphs of \emph{odd} degree, the problem can be solved by a \emph{deterministic} algorithm without communication (if some orientation is given).

Note that another example of a non locally checkable problem with such a separation between the randomized and deterministic complexities was given in \cite{GHK17}. Their problem consists in marking $(1+o(1))\sqrt{n}$ vertices of an $n$-cycle; the randomized version can also be solved in 0 rounds, while the deterministic version needs $\Omega(\sqrt{n})$ rounds. 

\subsection{Definitions}

A \emph{cut} in a graph $G$ is a bipartition $(A,B)$ of its vertex set $V(G)$. We usually refer to $A$ and $B$ as the left side and the right side of the cut, respectively. The \emph{size} of a cut $(A,B)$ is the number of edges with one end in $A$ and the other in $B$. The {\sc MaxCut} problem in a graph $G$ consists in finding a cut in $G$ whose size is maximum.

Given an oriented graph $G$, a \emph{directed cut} is again a bipartition $(A,B)$ of the vertex set of $G$, and the size of the directed cut $(A,B)$ is the number of arcs with their tail in $A$ and their head in $B$. The {\sc MaxDiCut} problem in an oriented graph $G$ consists in finding a directed cut in $G$ whose size is maximum.

\medskip

Our results in this paper mainly concern $d$-regular graph, i.e.\ graphs in which each vertex has degree $d$. When we refer to an oriented $d$-regular graph $G$, we mean that the underlying unoriented graph is $d$-regular (the out-degrees can be arbitrary).

\medskip

For an integer $k\ge 1$, the \emph{tower function} $\twr_k$ is the function defined as $\twr_1(x)=x$ and $\twr_{k}(x)=2^{\twr_{k-1}(x)}$ for $k\ge 2$. The \emph{iterated logarithm} of an integer $n$, denoted by $\log^*n$ is defined as 0 if $n\le 1$, and as $1+\log^*(\log n)$ otherwise (here and everywhere else in the paper, $\log$ denotes the base 2 logarithm). The following can be easily derived by induction on $k$.
\begin{cla}\label{tower_claim} For any $k,n\ge 1$:
$$\log^* (\twr_{k}(n))= k-1+\log^*(n)$$
\end{cla}

\section{Deterministic constant factor approximation in regular graphs}
\label{sec_up}

As mentioned in the introduction of this paper, Kawarabayashi and Schwartzman \cite{S17} provided a deterministic approximation algorithm running in $O(\log^* n)$ rounds for both problems studied here. In this section, we show using simple arguments based on bounds on Ramsey numbers that their bound is best possible.

\medskip

Let $[N]=\left\lbrace 1,\ldots ,N \right\rbrace$. The \emph{$q$-color Ramsey number} $r_k(n;q)$ is the minimum $N$ such that in any $q$-coloring of the $k$-element subsets of $[N]$, there is an $n$-element subset $S$ of $[N]$ such that all $k$-element subsets of $S$ have the same color (see~\cite{Ramsey17} for a recent survey on Ramsey numbers).

\begin{thm}[\cite{ES35,ER52}]\label{thm:ramsey}
 There exists $c>0$ such that for any positive integers $q$, $k$, and $n$, we have $r_k(n;q) \leq \twr_k(c\cdot n \cdot q \log q)$.
\end{thm}

We will also need two simple constructions of $d$-regular bipartite graphs.

\smallskip

First we assume that $d$ is even. We consider a cycle $C$ of size $n\ge 2d$, with $n$ even, and then add an edge between each pair of vertices that are at distance exactly $i$ in $C$ for every $i\in \left\lbrace 3,5,7, \ldots , d-1\right\rbrace$. This graph, which we denote by $C_n^d$, is certainly bipartite (the bipartition corresponds to the vertices at even distance from some arbitrary vertex in $C$, and the vertices at odd distance from this vertex). See figure \ref{cycle} for an example of this graph. By a slight abuse of notation, we say that two (or more) vertices of $C_n^d$ are \emph{consecutive} if they are consecutive in $C$. Similarly, when we refer to the \emph{clockwise order} around $C_n^d$, we indeed refer to the clockwise order around $C$.

\medskip

Assume now that $d$ is odd. We take two disjoint copies of $C_n^{d-1}$ and assume that the vertices of the cycle $C$ in the first copy are $u_1,u_2,\ldots,u_n$, in clockwise order, and the vertices of the cycle $C$ in the second copy are $v_1,v_2,\ldots,v_n$ in clockwise order. We then connect $u_i$ and $v_i$ by an edge, for any $1\le i \le n$ (and we say it what follows that $u_i$ and $v_i$ are \emph{matched}).
This graph, which we denote by $D_{2n}^d$, is clearly bipartite and $d$-regular, see Figure \ref{cycle_odd} for an example.

\begin{figure}[h!]
\centering
\begin{minipage}{0.5\textwidth}
\centering
\includegraphics[scale=0.4]{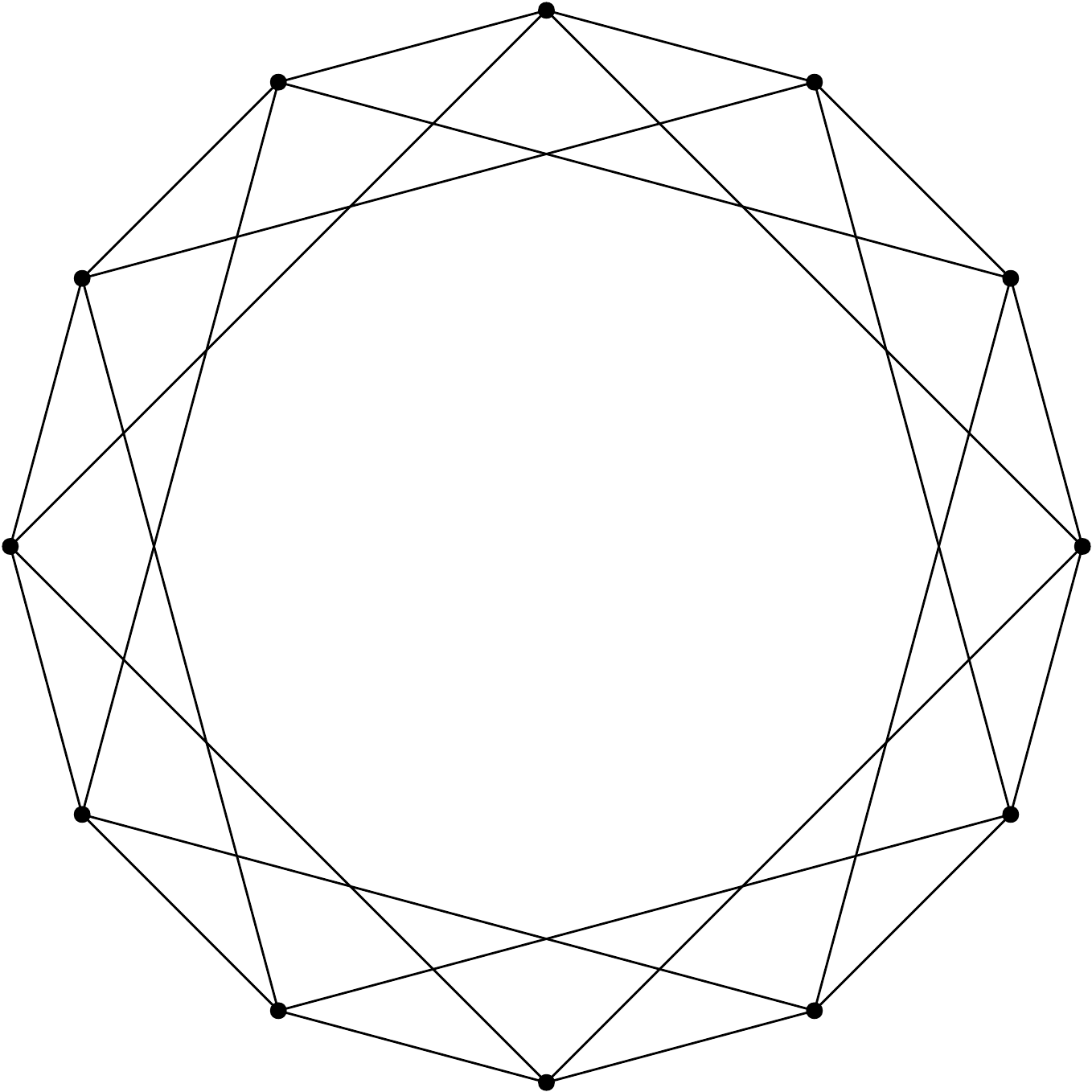}
\caption{$C_{12}^4$}
\label{cycle}
\end{minipage}%
\begin{minipage}{0.5\textwidth}
\centering
\includegraphics[scale=0.4]{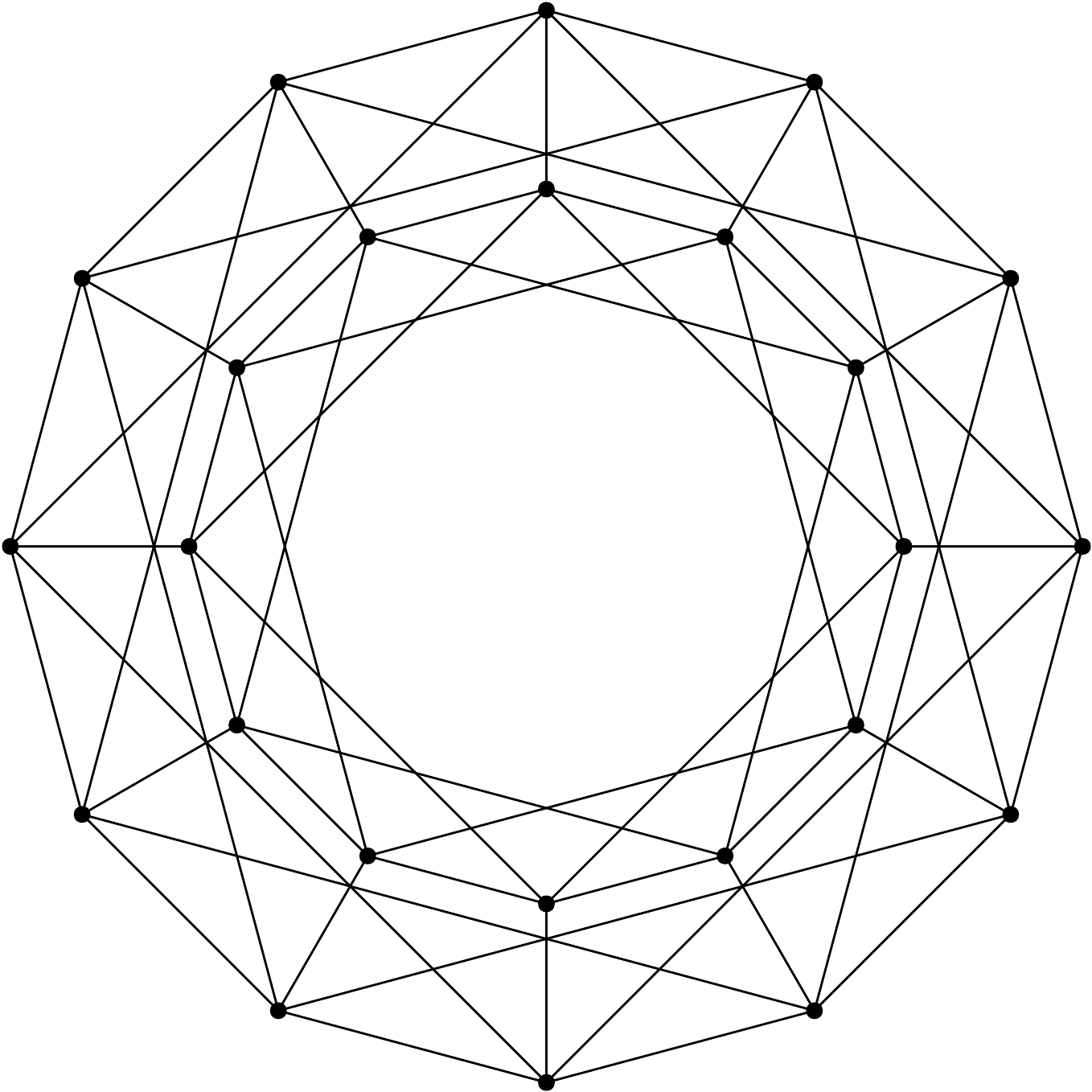}
\caption{$D_{24}^{5}$}
\label{cycle_odd}
\end{minipage}
\end{figure}

We are now ready to state the main result of this section. 


\begin{thm}\label{thm:lb}
Let $d\ge 2$ be a fixed integer.
\begin{itemize}
\item If $d$ is even, then any deterministic algorithm in the \textsf{LOCAL} model that guarantees a constant factor approximation for {\sc MaxCut} on the class of bipartite $d$-regular $n$-vertex graphs runs in $\Omega(\log^*n)$ rounds.
\item If $d$ is odd, then for any $\epsilon>0$, any deterministic $\left(\frac{1}{d}+\epsilon \right)$-approximation algorithm  in the \textsf{LOCAL} model for {\sc MaxCut} on the class of bipartite $d$-regular $n$-vertex graphs runs in $\Omega(\log^*n)$ rounds.
\end{itemize}
\end{thm}

Note that since the \textsf{LOCAL} model is less restrictive than the \textsf{CONGEST} model, this theorem is also valid in the \textsf{CONGEST} model.

\begin{proof} We prove the two cases of the theorem separately starting with $d$ even:

\smallskip

Let $d$ be an even integer and assume that there exists a local deterministic algorithm \texttt{A} running in $T$ rounds and guaranteeing a $\frac{1}{a}$-approximation for some fixed $a\ge 1$, with $T$ to be defined later. Note that since \texttt{A} runs in $T$ rounds, the output of a vertex $v$ in $C_n^d$ only depends on the IDs of the vertices at distance at most $T$ from $v$ in $C_n^d$, and thus at distance at most $(d-1)T$ of $v$ in $C$ (more precisely, since the subgraph induced by each ball of a given radius is the same, the output of a vertex only depends on the sequence of the IDs of its neighbors at distance $(d-1)T$ in $C$, in clockwise order).

\smallskip

Consider a subset $U=\{u_1,\ldots,u_\ell\}$ of $[n]$, with $u_1<\ldots <u_\ell$, and assume that $\ell$ consecutive vertices of $C_n^d$ (in clockwise order) have IDs $u_1,u_2,\ldots,u_\ell$, in this order. In what follows, we identify each vertex of  $C_n^d$ with its ID. We now set $r=2(d-1)T+1$, and start by proving the following claim (recall that $1/a$ is the approximation ratio of \texttt{A}):
\begin{cla}
If $\ell\ge 4adT$, at least $\frac{\ell d}{2}\left( 1-\frac{1}{2a}\right)$ edges of $C_n^d$ have both endpoints in $\tilde{U}=\left\lbrace u_{(r-1)/2+1}, \ldots, u_{\ell-(r-1)/2}\right\rbrace$.
\label{claim1}
\end{cla}
\begin{proof}
Count the edges by the length of the jump they make around the cycle $C$. For a jump of size $k$, there are more than $\ell-(r-1)-2k$ such edges. Summing for $k$ odd from $1$ to $d-1$ we obtain at least $$\frac{\ell d}{2}-\frac{d(r-1)}{2}-\frac{d^2}{2}$$ edges that have both endpoints in $\tilde{U}$. When $\ell\ge 4adT\ge 2a(r-1-d)$, this is at least $\frac{\ell d}{2}\left( 1-\frac{1}{2a}\right)$, which proves the claim.
\end{proof}

Consider some $r$-element subset $S=\{a_1,a_2,\ldots,a_r\}$ of $[n]$, with $a_1<a_2<\cdots <a_r$. We assign the IDs $a_1,a_2,\ldots,a_r$ (in this order) to $r$ consecutive vertices of $C_n^d$, in clockwise order, and look at the output of the vertex with ID $a_{(r+1)/2}$ (call it $v$) given by the algorithm \texttt{A}. Note that this output only depends on $\left(a_1,a_2,\ldots,a_r\right)$. If $v$ joins the left side of the cut (according to \texttt{A}), we color the set $S$  with color 0 and otherwise with color 1.

Consider again a subset $U=\{u_1,\ldots,u_\ell\}$ of $[n]$, with $u_1<\ldots <u_\ell$, and assume that $\ell$ consecutive vertices of $C_n^d$ (in clockwise order) have IDs $u_1,u_2,\ldots,u_\ell$. It follows from the definition of $\tilde{U}$ that if $v\in \tilde{U}$, all the vertices at distance at most $T$ from $v$ in $C_n^d$ are in $U$. This implies that if all $r$-elements subsets of $U$ are assigned the same color in the coloring defined above, all the vertices of $\tilde{U}$ choose the same side of the cut.

\smallskip

We can now apply Theorem~\ref{thm:ramsey} with $r$ as defined above and $\ell=\lceil 4adT\rceil$ satisfying the condition of Claim \ref{claim1}. Let $N=r_r(\ell;2) \leq \twr_r(c \ell)$ be given by Theorem~\ref{thm:ramsey}.  We now let $n$ be the smallest even integer which is greater than $2aN$ and consider $C_n^d$. Observe that by Claim~\ref{tower_claim}, $\log^* n=O(\log^* N)=O(r+\log^*(c \ell))=O(dT+\log^*(dT))=O(dT)$. Since $d$ is a fixed constant, it follows that we have $T=\Omega(\log^*n)$, as desired.

By Theorem~\ref{thm:ramsey}, there is an $\ell$-element subset $U_1$ of $[n]$ such that all $r$-elements subsets of $U_1$ have the same color. As long as there are more than $N$ remaining labels, we repeatedly apply Theorem~\ref{thm:ramsey} and thus find disjoint $\ell$-element subsets $U_1,U_2,\ldots ,U_k$ of $[n]$, with the same property (for each $U_i$, all $r$-elements subsets of $U_i$ have the same color), until $[n]-\bigcup_{i=1}^k U_i$ contains fewer than $N$ elements. We write each set $U_i$ as $\left\lbrace u_1^i,\ldots ,u_\ell^i\right\rbrace$, with $u_1^i<\cdots <u_\ell^i$.

\smallskip

Finally, we assign these IDs to consecutive vertices in clockwise order around $C_n^d$:
$$u_1^1,\ldots ,u_\ell^1, u_1^2, \ldots ,u^2_\ell,\ldots,u_1^k, \ldots, u_\ell^k $$

By Claim \ref{claim1}, the subgraph induced by each $\tilde{U_i}$ contains at least $\frac{md}{2}\left(1-\frac{1}{2a} \right)$ edges for all $1\leq i \leq k$, and it follows from Theorem~\ref{thm:ramsey}  and our coloring of the $r$-elements sets that for each $1\leq i \leq k$, all vertices in $\tilde{U_i}$ choose the same side of the cut.

\smallskip

Since $k \geq \frac{n-N}{\ell}$, by running algorithm \texttt{A} on this particular labelling of $C_n^d$, at least $$ \frac{n-N}{\ell}\frac{\ell d}{2}\left(1-\frac{1}{2a}\right) = \left(1-\frac{1}{2a}\right)\frac{d}{2} \left( n-N\right)>\frac{nd}{2} \left(1-\frac{1}{2a}\right)^2>\frac{nd}{2}\left(1-\frac{1}{a} \right)$$  edges are not in the cut (the last inequality uses the fact that $1/a \le 1$). Thus, there are strictly less than $\frac{nd}{2a}$ edges in the cut. Since $C_n^d$ is bipartite and $d$-regular, the optimal cut contains $\frac{nd}{2}$ edges (i.e.\ all the edges are in the cut). This proves that \texttt{A} cannot be a $\frac{1}{a}$-approximation, yielding a contradiction.

\medskip

Assume now that $d$ is odd and that an algorithm \texttt{B} achieves a $\left(\frac{1}{d}+\epsilon \right)$-approximation for some $\epsilon>0$ in $T$ rounds. We proceed as before except that instead of considering vertices one by one, we consider pairs of vertices $u_i,v_i$ that are matched in the construction of $D_{2n}^d$, the graph that will be used here. Similarly as before, a matched pair $(u_i,v_i)$ cannot see more than $T(d-1)$ labels away. Set $r=4T(d-1)+4$, and consider $r/2$ consecutive vertices (in clockwise order) $u_1,\ldots,u_{r/2}$ on the outer cycle of $D_{2n}^d$. For each $1\le i\le r/2$, let $v_i$ be the neighbor of $u_i$ on the inner cycle. This implies that $v_1,\ldots,v_{r/2}$ are consecutive (in clockwise order) on the inner cycle.

Fix an arbitrary $r$-element subset $\{a_1,\ldots,a_r\}$, with $a_1<a_2<\ldots<a_r$. For each $1\le i \le r/2$ assign the ID $a_i$ to $u_i$ and the ID $a_{r/2+i}$ to $v_i$. Since the sides of the cut chosen by $u_{r/4}$ and its neighbor $v_{r/4}$ are entirely determined by the set $\{a_1,\ldots,a_r\}$, we can color each set $\{a_1,\ldots,a_r\}$ with the pair $(x,y)\in \{ 0,1\}^2$ such that $u_{r/4}$ chooses side $x$ and $v_{r/4}$ chooses side $y$ (again we associate the left side of the cut with 0, and the right side of the cut with 1).

By exactly the same argument as in Claim \ref{claim1}, we can take $\ell$ even and large enough so that at least $\frac{\ell (d-1)}{2}\left(1-\frac{\alpha}{2}\right)$ edges are not in the cut in both copies of $C_n^{d-1}$, for any fixed $\alpha$. By Theorem~\ref{thm:ramsey} (with $q=4$), there is an integer $N=r_r(2 \ell;2)$ such that for all $n\geq N$, there is a $2 \ell$-element subset $U_1=\{a_1,\ldots, a_{2 \ell}\}$ of $[2n]$ with $a_1<\cdots<a_{2 \ell}$ such that all $r$-element subsets $X$ of $U_1$ have the same color. We take $n>\frac{N}{\alpha}$ even and repeatedly apply Theorem~\ref{thm:ramsey} as before, obtaining $2 \ell$-element subsets $U_1,U_2,\ldots,U_k$. We then assign the elements of each $U_i$ to consecutive vertices in clockwise order in $D^d_{2n}$ (the $\ell$ smaller elements of $U_i$ are assigned to the vertices of the outer cycle, and the $\ell$ larger elements are assigned to their corresponding neighbors in the inner cycle). As before, the fact that all
$r$-elements subsets of $U_i$ have the same color implies that on the portion of $D^d_{2n}$ corresponding to $U_i$, all the vertices of the outer cycle (except at the boundary) choose the same side $x$ of the cut, and all the vertices of the inner cycle (except at the boundary) choose the same side $y$ of the cut (but $x$ and $y$ might be different).

This ensures that on each copy of $C_n^{d-1}$, at least $$\frac{2n-N}{2 \ell}\frac{m(d-1)}{2}\left(1-\frac{\alpha}{2} \right) > \frac{n(d-1)}{2}\left(1-\frac{\alpha}{2}\right)- \frac{n(d-1)}{2}\left(1-\frac{\alpha}{2}\right)\frac{\alpha}{2}>\frac{n(d-1)}{2}\left(1-\alpha \right)$$ edges are not in the cut after running algorithm \texttt{B}. It follows that in $D_{2n}^d$, at least $n(d-1)(1-\alpha)$ edges are not in the cut. Since $D_{2n}^d$  is bipartite and contains exactly $nd$ edges, this shows there are less than $n+n(d-1)\alpha$ edges in the cut, which is a $(\frac{1}{d}+(d-1)\alpha)$-fraction. Setting $\alpha=\frac{\epsilon}{d-1}$, this fraction is less than $\frac{1}{d}+\epsilon$, which is a contradiction.
\end{proof}

\medskip

A direct consequence of our theorem is the following corollary that matches the round complexity obtained by Kawarabayashi and Schwartzman \cite{S17}:

\begin{cor}
Deterministic constant factor approximation on general graphs for {\sc MaxCut} in the \textsf{LOCAL} model requires $\Omega(\log^* n)$ rounds.
\end{cor}

\subsection{Directed cut} In this section, we consider the similar problem {\sc MaxDiCut} where edges are oriented and we only count the edges going from the left side of the cut to the right side. We can prove similar bounds on the quality of the solution one can hope to achieve by simply orienting our lower bound graphs $C^d_n$ and $D^d_{2n}$: we will define $\overrightarrow{C^d_n}$ as the same graph as $C^d_n$ where we orient all the edges in clockwise order. Similarly, $\overrightarrow{D^d_{2n}}$ is obtained from $D^d_{2n}$ by orienting all the edges in clockwise order on both the inner and outer cycle, and all the edges in the remaining perfect matching from the outer cycle to the inner cycle. We can again apply Ramsey's theorem as in the proof of Theorem \ref{thm:lb} to obtain the following result :
\begin{thm}
 Let $d>0$ be a fixed integer.
\begin{itemize}
\item If $d$ is even, any deterministic algorithm that guarantees a constant factor approximation for {\sc MaxDiCut} on the class of $d$-regular bipartite $n$-vertex oriented graphs requires $\Omega(\log^*n)$ rounds in the \textsf{LOCAL} model.
\item If $d$ is odd, then for any $\epsilon>0$, any deterministic $\left(\frac{2}{d}+\epsilon \right)$-approximation of {\sc MaxDiCut} on the class of $d$-regular bipartite $n$-vertex oriented graphs requires $\Omega(\log^*n)$ rounds in the \textsf{LOCAL} model.
\end{itemize}
\label{thm:oriented}
\end{thm}

We note a slight difference with Theorem~\ref{thm:lb} in the case where $d$ is odd. In Theorem~\ref{thm:oriented} the approximation ratio is only bounded by $\frac{2}{d}$, instead $\frac{1}{d}$. This happens because with our definition of $\overrightarrow{D^d_{2n}}$, one can check that the optimal directed cut is of size $\frac{nd}{4}=\frac{m}{2}$ instead of $m$ in the undirected case.

\section{Matching the approximation ratio when $d$ is odd} 
\label{sec:low}

\subsection{Weak-coloring}

In a landmark paper, Naor and Stockmeyer~\cite{N93} addressed the issue of what can or cannot be computed locally. In particular, they proved one result that is relevant in our case. 

A \textit{weak coloring} of a graph is a coloring of its vertices such that each vertex has at least one neighbor with a different color. Observe that a weak coloring using only $2$ colors is a $\frac{1}{d}$-approximation of the {\sc MaxCut} problem when the graph is $d$-regular. Let $O_d$ be the class of graphs of maximum degree $d$ where the degree of every vertex is odd. Naor and Stockmeyer proved the following theorem.
\begin{thm}[\cite{N93}]
There is a constant $b$ such that, for every $d$, there is a deterministic algorithm with round complexity $\log^* d + b$ in the \textsf{CONGEST} model that solves the weak $2$-coloring problem in the class $O_d$.
\end{thm}

As discussed above, this result directly implies that one can produce a local deterministic $\frac{1}{d}$-approximation of the {\sc MaxCut} problem on $d$-regular graphs. However, the result given here is much stronger than what we are looking for as in this case \emph{every} vertex has at least one incident edge in the cut. A natural question is whether a faster algorithm (of round complexity that does not depend on $d$) exists for the {\sc MaxCut} problem on $d$-regular graphs with $d$ odd. In the next section, we prove that such an algorithm exists.

\subsection{A simpler and faster algorithm}

Consider the following algorithm: every vertex $v$ collects the list of \textrm{IDs} of its neighbors, then $v$ chooses its side of the cut depending on whether the median value of this list is higher or lower than its own \textrm{ID}. We call this algorithm the \textit{median algorithm}. It runs in a single round and we prove the following theorem: 

\begin{thm}\label{thm:median1}
When the input is a $d$-regular graph on $n$ vertices, with $d$ odd, the median algorithm finds in 1 round a $\frac{1}{d}$-approximation for the {\sc MaxCut} problem in the \textsf{CONGEST} model.
\end{thm}

We will actually give two different proofs of this result (i.e.\ Theorem~\ref{thm:median1} will be a direct consequence of Theorem~\ref{thm:median2}, which we proved next, but also of Theorem~\ref{thm:mediand}, which will be proved in Section~\ref{sec:oricut}).

\begin{thm}\label{thm:median2}
When the input is a $d$-regular graph on $n$ vertices, with $d$ odd, the median algorithm outputs in 1 round (in the \textsf{CONGEST} model) a cut of size at least $\tfrac{n}2+\tfrac{(d-1)(d+1)}4$.
\end{thm}

\begin{proof}
Let $G=(V,E)$ be a $d$-regular graph. In the proof, we say that a vertex is colored $0$ or $1$ by the median algorithm if it is assigned to the left side or the right side of the cut (respectively). We define the boundary of a subgraph of $G$ as the set of edges of $G$ that have exactly one endpoint in this subgraph. For every subset $A\subset V$ we denote by $G[A]$ the subgraph induced by $A$ and $\partial A$ the boundary of $G[A]$. 

We now orient each edge of $G$ from the vertex of lower \textrm{ID} to the vertex of higher ID: it can be observed that the median algorithm assigns color $0$ to vertices that have more outgoing edges than ingoing edges and side $1$ when it is the opposite. We also note that this orientation is acyclic, which is the key property that will be used in this proof. 

A \emph{monochromatic component} is a connected component of the subgraph of $G$ induced by one of two sides of the cut. We now prove the following two simple claims.

\begin{cla}
After running the median algorithm, any subset $A$ of a monochromatic component contains a vertex with at most $\frac{d-1}{2}$ neighbors in $A$.
\label{claim_1}
\end{cla}

Assume for the sake of contradiction that any vertex of $A$ has at least $\tfrac{d+1}2$ neighbors in $A$. Since the orientation of $G[A]$ is acyclic, there must be a sink and a source. It follows that one has outdegree at least $\tfrac{d+1}2$, and the other has indegree at least $\tfrac{d+1}2$. By definition of the median algorithm, the source and the sink must be on different sides of the cut, which contradicts the fact that $A$ is monochromatic.
This concludes the proof of Claim~\ref{claim_1}.

\begin{cla}
After running the median algorithm, for every monochromatic component $A$ of size $k$,  $\partial A$ contains at least $k+\tfrac{(d-1)(d+1)}4$ edges if $k\ge \tfrac{d+1}2$, and at least $\tfrac{d+1}2 \cdot k$ edges otherwise.
\label{claim_2}
\end{cla}

Let $A$ be a monochromatic component of $V$ of size $k$. Following Claim \ref{claim_1}, one can order the vertices $v_1,\ldots, v_k$ of $A$ such that for all $1\leq i \leq k$, $v_i$ has at most $\frac{d-1}{2}$ neighbors in $\left\lbrace v_{i+1},\ldots,v_k \right\rbrace$. If $k\ge \tfrac{d+1}2$, it follows that there are at most $\frac{d-1}{2}\cdot k-\tfrac{(d-1)(d+1)}8$ edges in $G[A]$. But $G$ is $d$-regular, therefore we have in this case:
$$|\partial A| + 2\left(\tfrac{d-1}{2}\cdot k -\tfrac{(d-1)(d+1)}8\right)\geq d\cdot k$$ which implies $|\partial A| \geq k+\tfrac{(d-1)(d+1)}4$.

If $k\le \tfrac{d+1}2$, $G[A]$ contains at most ${k \choose 2}\le k \cdot \tfrac{d-1}4$ edges, and a similar computation shows that $|\partial A| \geq k(d-\tfrac{d-1}2)\ge k\cdot \tfrac{d+
  1}2$. This concludes the proof of Claim~\ref{claim_2}.

\medskip

Finally, let $X$ be the larger side of the cut output by the median algorithm, i.e. $|X|\ge \tfrac{n}2$. Observe that the boundary $\partial X$ is the union of the boundaries of the connected components of $G[X]$ (since there are no edges between two such components). If at least one of these components has size at least $\tfrac{d+1}2$, then it follows from Claim \ref{claim_2} that $|\partial X|\ge \tfrac{n}2+\tfrac{(d-1)(d+1)}4$, as desired. Otherwise all connected components of $G[X]$ have size at most $\tfrac{d-1}2$, and it follows from Claim \ref{claim_2} that $|\partial X|\ge \tfrac{n}2 \cdot \tfrac{d+1}2=\tfrac{n}2+\tfrac{d-1}2\cdot \tfrac{n}2\ge \tfrac{n}2+\tfrac{(d-1)(d+1)}4$, since $n\ge d+1$ (recall that $G$ is $d$-regular, and thus contains at least $d+1$ vertices). This concludes the proof of Theorem~\ref{thm:median2}.
\end{proof}

Since a $d$-regular graph has $\tfrac{dn}2$ edges, we conclude that the median algorithm gives a $\tfrac1d$-approximation for the maximum cut when $d$ is odd, which proves Theorem~\ref{thm:median1}.

\medskip

Figure~\ref{fig:ex} gives an example of labelling of $D_{2n}^d$ for which the median algorithm gives a cut of size $\frac{n}{2} + (d-2)^2 + 1$. This shows that our analysis of the median algorithm in Theorem~\ref{thm:median1} is close to being best possible. 

\medskip

\begin{figure}[htb]
  \includegraphics[scale=0.5]{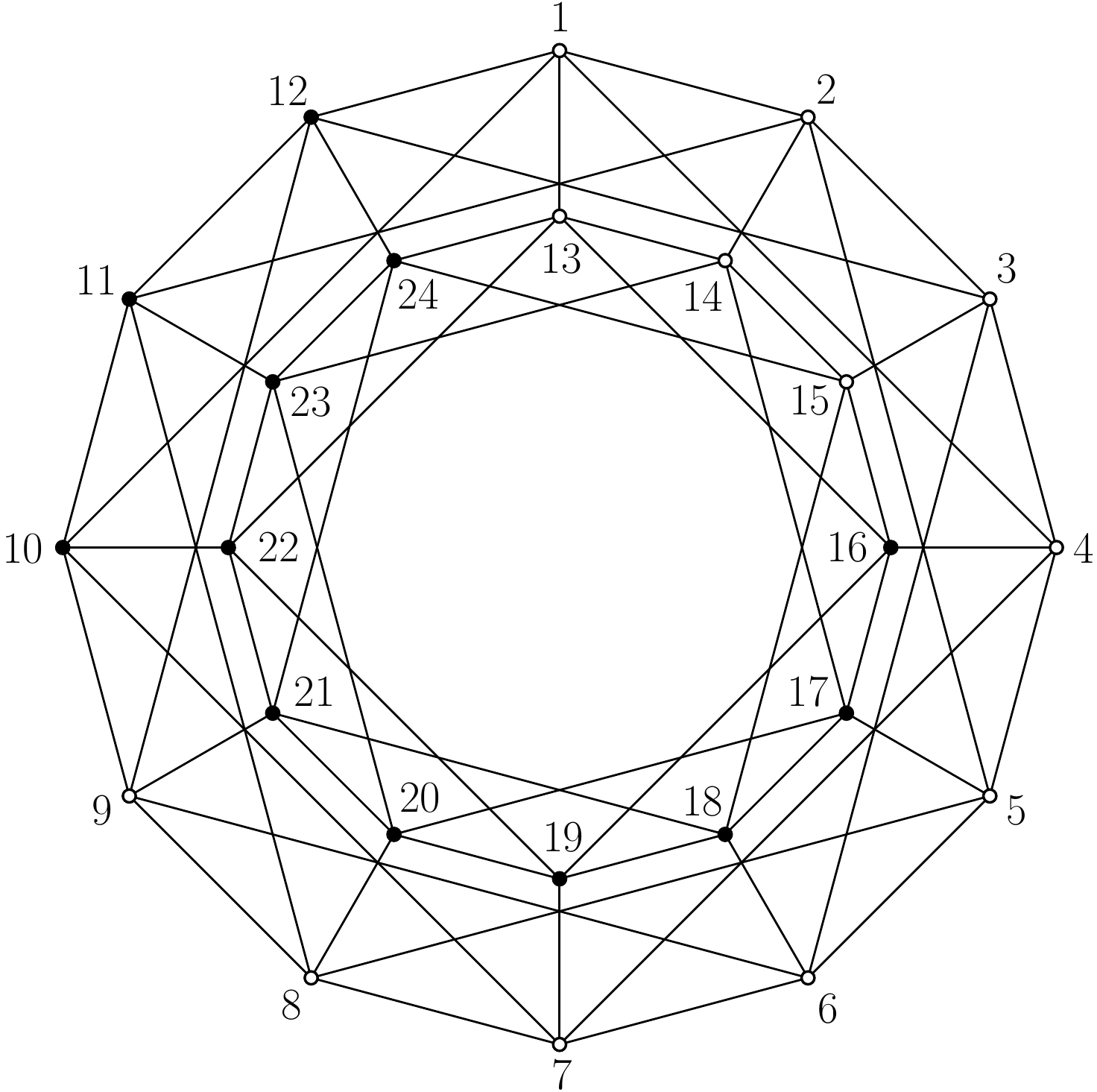}
  \centering
\caption{Extremal labelling of $D_{24}^5$ for the \textit{median} algorithm}\label{fig:ex}
\end{figure}

Another interesting aspect of Theorem~\ref{thm:median2} is that it shows that in (the second item of) Theorem~\ref{thm:lb}, it is crucial that $d$ is a fixed constant (independent of $n$). Indeed, if $d=\Omega(\sqrt{n})$, then  $\tfrac{n}2+\tfrac{(d-1)(d+1)}4\ge (1+\Omega(1))\tfrac{n}2$ and thus the median algorithm achieves a $\tfrac{1+\epsilon}{d}$-approximation, for some $\epsilon>0$. This is impossible when $d$ is a constant, as shown by Theorem~\ref{thm:lb}.

\medskip

The median algorithm is based on finding an (acyclic) orientation of the input graph. Here, we do it by simply orienting the edges from the end with lower ID to the end with higher ID. This costs a single round of communication, with messages of size $\log n$ (since vertices have to send their ID to their neighbors). It follows that in the more restricted $\textsf{CONGEST}(b)$ model, where messages have size at most $b$, our algorithm takes $\tfrac{\log n}b$ rounds (here and in the remainder, we omit floors and ceilings whenever they are not necessary in the discussion). In particular, if only messages of size 1 are allowed, our algorithm takes $\log n$ rounds. 

A natural question is whether this can be improved. We now argue that it cannot be improved in general if the algorithm is based on some orientation in the graph. Consider the case where $G$ contains an isolated edge $uv$ (two adjacent vertices $u,v$ of degree 1 in $G$) and we want to construct some orientation of $G$ (and in particular of the edge $uv$) in the $\textsf{CONGEST}(1)$ model (that is with messages of size 1). It seems that the argument below is not original, but we have not been able to find a written source.

Observe that at each round of communication, the message sent by each of $u,v$ only depends on its ID and the bits received from its neighbor at previous steps. At the first round, at least half of the IDs (call this set $S_1$) would send the same bit, say $b_1$, to their neighbor. At the next round, at least half of the IDs of $S_1$, upon receiving $b_1$, would sent the same bit to their neighbor, say $b_2$. We continue this process by constructing sets $S_i$ and bits $b_i$ for any $1\le i < \log n$ as above (except that for the final round, we define $b_i$ as the bit output by the vertices, instead of the bit sent to the neighbor). If we use less than $\log n$ rounds of communication, we can find two distinct IDs in the last set $S_i$ such that if we assign these IDs to $u$ and $v$, these two vertices will output the same bit $b_i$, and therefore they will not be able to deterministically agree on an orientation of the edge $uv$. Actually the result holds even if randomization is allowed and we want $u$ and $v$ to agree on some orientation of the edge $uv $ with high probability.

\medskip






This remark leads us to a similar result for approximating {\sc MaxCut} in regular graphs. We prove the following:
\begin{thm}\label{thm:congest1} Let $D^d=\left\lbrace D_{2n}^d, n>0 \right\rbrace$ for $d$ odd.
Any deterministic constant factor approximation of {\sc MaxCut} on the class $D^d$ requires at least $(1-o(1))\log n$ rounds in the $\textsf{CONGEST}(1)$ model.
\end{thm}
\begin{proof}
Assume that an algorithm \texttt{A} achieves a $\frac{1}{a}$-approximation ($a>0$) on the class of $d$-regular graphs in at most $(1-\alpha) \log n$ rounds (with $\alpha>0$).

Assume first that we have a bipartite $d$-regular with $m$ edges and consisting of $\ell=n^{1-\alpha}$ connected components of size $k=n^\alpha$, labelled $C_1,C_2,\ldots, C_\ell$, and in which each vertex has the same ``view'' at distance $\log n$ (i.e.\ the balls of radius $\log n$ centered in each of the vertices are isomorphic). We argue, in the same way as above, that we can choose $\tfrac{n}{2^{(1-\alpha)\log n}}=n^{\alpha}=k$ labels such that during $(1-\alpha)\log n$ rounds, every vertex in $C_1$ outputs the same bit. More precisely, for each round $1\le i \le (1-\alpha)\log n$, each vertex of $C_1$ outputs the same bit $b_i$. We proceed similarly for $C_2,\ldots,C_{\ell'}$ with $\ell'=\ell(1-\tfrac1{2a})$: this is possible since before labelling $C_{i}$ ($i\leq \ell'$) there are at least $n-k\ell'=\frac{n}{2a}=\Omega(n)$ available labels.

Hence, we give a labelling of $C_1,\ldots,C_{\ell'}$ such that during $(1-\alpha)\log n$ rounds (and in particular, at the end of the algorithm), all the vertices in each given connected component output the same bit.  In particular  no edge of $C_1,\ldots,C_{\ell'}$ appears in the cut, which implies that at most $\frac{m}{2a}$ edges appear in the cut. Since every component is bipartite, the maximum cut contains $m$ edges, which contradicts the hypothesis that \texttt{A} achieves a $\frac{1}{a}$-approximation.

\medskip 

We then show that this proof can be adapted to the case of $C_n^d$ or $D_{2n}^d$ (depending on the parity of $d$). Partition the graph into sets of vertices that appear consecutively in the clockwise order around the cycle(s): $C_1,\ldots,C_\ell$ (each of size $n^\alpha$). As before, we can label $C_1,\ldots,C_{\ell'}$ such that during $(1-\alpha)\log n$ rounds, the output bit of every vertex of $C_i$ is the same (for every $1\le i\le \ell'$). When \texttt{A} runs on the graph, a small perturbation of the output bit may appear on the boundary of each $C_i$, therefore the output bit of vertices near the boundary is not guaranteed to be the same anymore. However, this does not hurt us since we run the algorithm for $O(\log n)$ rounds therefore the perturbation can reach at most $d\log n$ vertices in every $C_i$. Therefore, in each $C_i$ at least $n^\alpha-O(\log n)$ vertices have the same output during $(1-\alpha)\log n$ rounds (and in particular, the same output at the end of the algorithm). This implies there are at least $\ell'(n^{\alpha}-O(\log n)) = n(1-\tfrac1{2a}) - o(n)$ vertices that have the same output as their neighbors at the end of the algorithm. It follows that the cut output by the algorithm has size at most $\tfrac{m}{2a}+o(m)$, while the maximum cut has size $m$, which contradicts the definition our initial assumption that that \texttt{A} achieves a $\frac{1}{a}$-approximation.
\end{proof}

Note that the randomized, high probability version of Theorem~\ref{thm:congest1} does not hold. More precisely, we can show that the trivial randomized $\frac{1}{2}$-approximation indeed produces a constant factor approximation with high probability on the class of regular graphs in the $\textsf{CONGEST}(0)$ model. This is certainly a classic result but we have not been able to find it in the literature. 
\begin{thm}[folklore]
For any $\epsilon>0$, the folklore algorithm produces a $\frac{1-\epsilon}{2}$-approximation with high probability on the class of $n$-vertex graphs with degrees bounded by a constant, and $m=\Omega(n)=\Omega(1/\epsilon^2)$ edges (and in particular in the class of $d$-regular graphs of sufficiently large size).
\end{thm}
\begin{proof}
  Let $G=(V,E)$ be a graph of maximum degree $d$ containing $m$ edges. By Vizing's Theorem, $G$ has a $(d+1)$-edge-coloring, i.e.\  a partition of its edge-set into $d+1$ matchings $M_1,\ldots,M_{d+1}$. Assume that $|M_1|\ge \cdots \ge |M_{d+1}|$, and discard all the matchings $M_i$ such that $|M_i|\le \epsilon' m/d$, for some $\epsilon'>0$ whose value will be fixed later in the proof. Note that the remaining matchings $M_1,\ldots,M_k$ satisfy $|\bigcup_{i=1}^kM_i|\ge (1-\epsilon')m$. Recall that the folklore algorithm assigns each vertex to one of the two sides of the cut, uniformly at random. Note that for each matching $M_i$, and for any two edges $e,f\in M_i$, the events that $e$ and $f$ are in the cut are independent.

We now recall the following Chernoff bound (see e.g.\ Chapter 5 in~\cite{MR02}): For any $0\le t\le np$, $\Pr(|\mbox{BIN}(n,p)-np|> t)\le 2\exp(-t^2/3np)$, where $\mbox{BIN}(n,p)$ denotes the binomial distribution with parameters $n$ and $p$.
  Thus, for each $1\le i \le k$, it follows that with probability at least $1-n^{-2}$, at least $\tfrac12\,|M_i|-c\sqrt{|M_i|}$ edges of $M_i$ are in the cut output by the algorithm, for some absolute constant $c>0$. Using the Union bound, with probability at least $1-1/n$ the cut output by the algorithm contains at least $$\sum_{i=1}^k (\tfrac12\,|M_i|-c\sqrt{|M_i|})\ge \tfrac{m}2 (1-\epsilon') -c\cdot \sqrt{k} \cdot\sqrt{\sum_{i=1}^k|M_i|}\ge \tfrac{m}2 (1-\epsilon')-c\cdot d\cdot \sqrt{m}=\tfrac{m}2 (1-\epsilon'-\tfrac{2cd}{\sqrt{m}})$$ edges. Setting $\epsilon'=\epsilon-\tfrac{2cd}{\sqrt{m}}$ yields the desired result (recall that $m=\Omega(1/\epsilon^2)$ and thus such an $\epsilon'>0$ exists).
\end{proof}

\subsection{Directed cuts}\label{sec:oricut}

Given a bipartition $(V_1,V_2)$ of an oriented graph $G$, the set of arcs oriented from $V_1$ to $V_2$ (the \emph{directed cut} from $V_1$ to $V_2$) is denoted by $\overrightarrow{E}(V_1,V_2)$. The maximum cardinality of a directed cut in $G$ is denoted by $\mdc(G)$. 

\smallskip

Let $G$ be an oriented graph. For each vertex $v$, we define the \emph{deficit} of $v$ as $\delta(v)=d^+(v)-d^-(v)$, where $d^+(v)$ and $d^-(v)$ denote the out-degree and in-degree of $v$, respectively. We define the \emph{sign} of a vertex $v$ as the sign of $\delta(v)$, and we say that that a vertex is \emph{positive} or \emph{negative} accordingly. The set of positive vertices is denoted by $V^+$ and the set of negative vertices is denoted by $V^-$. Note that if all the vertices of $G$ have odd degree (in particular if $G$ is $d$-regular with $d$ odd), then every vertex is positive or negative and this case $V^+,V^-$ form a bipartition of the vertex set $V$ of $G$.

Note that the median algorithm described in the previous subsection can be rephrased as: find an acyclic orientation of $G$ and then choose the cut $(V^+,V^-)$ with respect to this orientation.
Our second proof of Theorem~\ref{thm:median1} will be a direct consequence of the following general result (which proves that not only the cut, but also the \emph{directed} cut between $V^+$ and $V^-$ has size at least $n/2$, and that the original orientation does not need to be acyclic).

\begin{thm}\label{thm:mediand}
  Let $G$ be an $n$-vertex oriented $d$-regular graph with $d$ odd, and let $V^+$ and $V^-$ be defined as above. Then the directed cut $\overrightarrow{E}(V^+,V^-)$ contains at least $\max\{\tfrac{n}2,\tfrac{2}{d+1/d} \cdot\mdc(G)\}$ arcs.
\end{thm}

\begin{proof}
  We write $\cut=|\overrightarrow{E}(V^+,V^-) |$ and $\opt=\mdc(G)$, and set $D=\sum_{v\in V^+} d^+(v) + \sum_{v\in V^-} d^-(v)\geq n\cdot \frac{d+1}{2}$. Observe that $D$ counts the number of arcs of $G[V^+]$ and $G[V^-]$ once, while the arcs of $\overrightarrow{E}(V^+,V^-)$ are counted twice. Since $G$ contains $dn/2$ arcs, it follows that 
\begin{equation}\label{eq1}
\cut\geq D-\frac{dn}{2} \geq \frac{n}{2}.
\end{equation}

This proves that the directed cut output by the algorithm has size at least $\frac{n}{2}$, which readily implies Theorem~\ref{thm:median1}.

\begin{figure}[htb]
\centering
\includegraphics[scale=1]{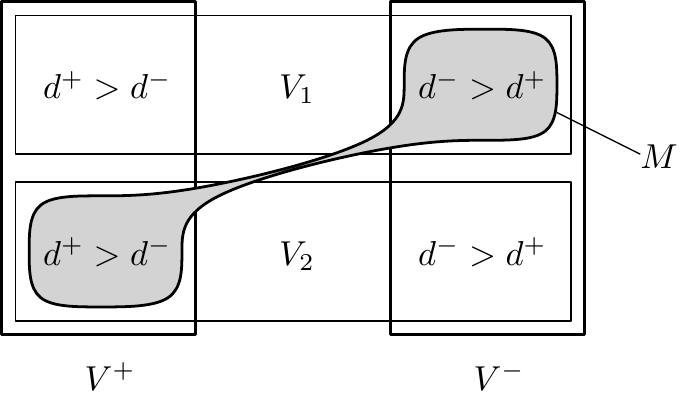}
\caption{Sets $V^+$, $V^-$, $V_1$, $V_2$, and $M$.}
\label{fig:dessincut}
\end{figure} 


\smallskip

We now consider an optimal directed cut $\overrightarrow{E}(V_1,V_2)$ of $G$ (i.e.\ a directed cut of cardinality $\opt=\mdc(G)$), and define $M$ as the set of vertices $(V_1\cap V^-)\cup(V_2\cap V^+)$ (see Figure~\ref{fig:dessincut} for an illustration). Note that each arc of $\overrightarrow{E}(V_1,V_2)$ which is not incident to a vertex of $M$ is also an arc of $\overrightarrow{E}(V^+,V^-)$. Since the vertices of $V_1\cap V^-$ have out-degree at most $\tfrac{d-1}2$ and the vertices of $V_2\cap V^+$ have in-degree at most $\tfrac{d-1}2$, we have
\begin{equation}\label{eq2}
\cut \geq  \opt- \frac{d-1}{2}\cdot |M|.
\end{equation}

Now observe that 

$$2|\overrightarrow{E}(V_1,V_2)| \leq \sum_{v\in V^+\setminus M} d^+(v) + \sum_{v\in V^-\setminus M} d^-(v) + \sum_{v\in V^-\cap M} d^+(v) + \sum_{v\in V^+\cap M} d^-(v)$$ 
$$= D + \sum_{v\in V^-\cap M} (d^+(v)-d^-(v)) + \sum_{v\in V^+\cap M} (d^-(v)-d^+(v) )
\leq D - |M|.$$

This implies
\begin{equation}\label{eq3}
2\cdot\opt \leq D-|M|.
\end{equation}

It follows from (\ref{eq2}) that $|M|\geq \frac{2}{d-1}(\opt-\cut)$, which we can plug into (\ref{eq3}) to obtain:
$$\tfrac{2d}{d-1}\,\opt\leq D+\tfrac{2}{d-1}\,\cut.$$

It directly follows from  (\ref{eq1}) that $D\leq \cut+\frac{dn}{2}\leq (d+1) \,\cut$ and plugging it into the previous inequality, we obtain:

$$\tfrac{2d}{d-1}\,\opt\leq (d+1+\tfrac{2}{d-1})\,\cut=\tfrac{d^2+1}{d-1}\,\cut,$$

and finally:

$$\frac{\cut}{\opt}\geq \frac{2d}{d^2+1}=\frac{2}{d+1/d},$$ as desired.
\end{proof}

From now on, we call the 0-round algorithm resulting from Theorem~\ref{thm:mediand} the \emph{oriented median algorithm}.
The factor $\frac{2}{d+1/d}$ might seem a little surprising, but it turns out to be sharp, in the following sense: there are $d$-regular oriented graphs for which the oriented median algorithm outputs a cut of size precisely $\frac{2}{d+1/d}\opt$. To see this, take $n$ to be a multiple of $4d$, and take 4 sets of vertices $A,B,C,D$ as in Figure~\ref{fig:abcd}. Each set is an independent set, and its size is $n$ times the fraction indicated in the figure (for instance $A$ and $B$ both contain $\tfrac{d+1}{4d}\cdot n$ vertices). The arc labelled $\tfrac12$ between $A$ and $B$ indicates that we add $\tfrac12\cdot n$ arcs joining $A$ to $B$, and similarly for the arcs joining $A$ and $D$, and the arcs joining $B$ and $C$). It can be checked that the number of arcs incident to each set is precisely $d$ times the size of the set, so the graph can be made $d$-regular (and we can make sure that the out-degree of each vertex is equal to the average out-degree of its part, for instance the vertices of $A$ have out-degree $\tfrac{4d}{d+1}\cdot (\tfrac{(d-1)^2}{8d}+\tfrac12)=\tfrac{d+1}2$, so they lie in $V^+$). It can also be checked that the directed cut output by the algorithm, $\overrightarrow{E}(A\cup D,B\cup C)$, has cardinality $n/2$ (the arcs joining $A$ to $B$), while the optimal directed cut $\overrightarrow{E}(A\cup C,B\cup D)$ contains $\tfrac{d^2+1}{2d}\cdot \tfrac{n}2=\tfrac{d+1/d}2\cdot \tfrac{n}2$ arcs.

\begin{figure}[htb]
\centering
\includegraphics[scale=0.8]{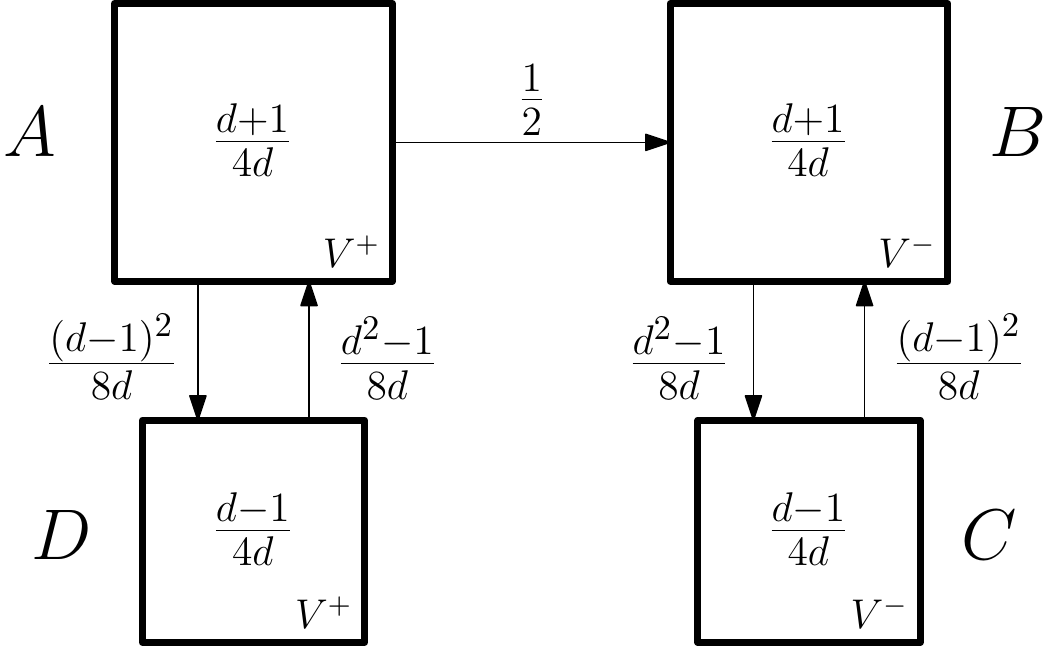}
\caption{An example showing the sharpness of the analysis of Theorem~\ref{thm:mediand}}
\label{fig:abcd}
\end{figure} 

So the problem does not come from the analysis of the algorithm, but rather from the algorithm itself, which is unable to detect the kind of structure depicted in Figure~\ref{fig:abcd}.
\bigskip

To overcome this issue and close the gap with the $\frac{2}{d}$ bound, one might be tempted to consider local improvements. In the following, a vertex will be \textit{stable} if it has at least one neighbor on the other side of the cut. Otherwise it will be \textit{unstable}. We now consider the following simple algorithm: at every round, every unstable vertex changes side. The algorithm stops when all vertices are stable. As we can see, the running time (and even the termination) of this algorithm is highly dependent on the starting point: for instance if all vertices start on the same side of the cut, then the algorithm never ends. When we perform one round of this algorithm, we say that we perform a \textit{flip} (as this algorithm can be seen as a variant of the well known FLIP algorithm that is further discussed in the conclusion). 

\bigskip

Even though this algorithm may never end, we will prove shortly that if we take as starting point the cut given by the oriented median algorithm (that gives a $\frac{2}{d+1/d}$-approximation in 0 rounds) and perform $2$ flips, we then improve slightly on the approximation ratio of $\frac{2}{d+1/d}$.

We first explain some useful properties of stability, as defined above.

\begin{cla}
Once a vertex is stable, it remains stable after any number of flips.
\end{cla}
\begin{proof}
Simply notice that if a vertex $u$ is stable, then it has a neighbor $v$ on the other side of the cut and $v$ must be stable too. If $u$ or $v$ become unstable then one of them becomes unstable while the second one is on the other side of the cut, which is impossible.
\end{proof}
\begin{cla}\label{cla:inc}
  Once an edge is in the directed cut, it remains in the directed cut after any number of flips.
\end{cla}
\begin{proof}
Any edge in the directed cut joins two stable vertices, and thus remains in the cut after any number of  flips.
\end{proof}

We define $\cut_i$ to be the number of edges in the directed cut  after $i$ flips. We take the notation defined in the proof of Theorem~\ref{thm:mediand}: $V^+$ and $V^-$ are the sets of vertices of positive and negative deficit, respectively, $\cut_0=\cut=|\overrightarrow{E}(V^+,V^-)|$ is the size of the dicut given by the oriented median algorithm (running in 0 rounds), $\opt=\mdc(G)$ is the size of the maximum dicut, $D=\sum_{v\in V^+} d^+(v) + \sum_{v\in V^-} d^-(v)$, and $M$ is the set of vertices whose side differ in $\overrightarrow{E}(V^+,V^-)$ and in some fixed maximum dicut $(V_1,V_2)$.

By Claim \ref{cla:inc}, we have that $\cut_j\geq \cut_i$ for any $j\ge i$. Using this, a simple modification of the proof of Theorem~\ref{thm:mediand} shows the following:

\begin{cla}\label{cla:formula}
Assume we have the following inequalities for some $\alpha,\beta \in [0,1]$:
$$\cut_j\geq \opt - \frac{d-1}{2}\cdot |M| + \alpha\cdot |M|$$
$$2\cdot \opt\leq D-|M| - \beta \cdot |M|$$
then $$\frac{\cut_j}{\opt}\geq f_d(\alpha,\beta):=\frac{d-2\alpha+\beta}{d^2/2 -\alpha\cdot (d+1) +\beta+1/2}$$
\end{cla}
The following claim immediately holds as well:
\begin{cla}\label{cla:f_inc}
For any, $d\geq 3$, $\alpha,\beta \in [0,1]$ such that $\alpha+\beta>0$, $f_d(\alpha,\beta)>f_d(0,0)=\frac{2}{d+1/d}$.

More precisely, for any $y\in (0,1)$, $$\inf_{\alpha,\beta \in [0,1],\alpha+\beta \geq y} f_3(\alpha,\beta)=f_3(0,y)=\frac{3+y}{5+y}$$ and for any $d\geq 5$, 
$$\inf_{\alpha,\beta \in [0,1],\alpha+\beta \geq y} f_d(\alpha,\beta)=f_d(y,0)=\frac{d-2y}{d^2/2-y(d+1)+1/2}$$
\end{cla}
\begin{proof}
Notice that, for any $d\geq 3$, $\alpha,\beta \in (0,1)$, $\frac{\partial}{\partial \alpha} f_d(\alpha,\beta)=\frac{4(d-1)(\beta+1)}{(d^2-2\alpha\cdot (d+1)+2\beta + 1)^2}>0$ and $\frac{\partial}{\partial \beta} f_d(\alpha,\beta)=\frac{2(d-1)(d-2\alpha+1)}{(d^2-2\alpha\cdot (d+1)+2\beta + 1)^2}>0$. This immediately proves the first claim.

For the second claim, by previous calculations we can set $\beta = y-\alpha$ and compute:
$$\frac{d}{d\alpha}f_d(\alpha,y-\alpha) = -\frac{2(d-1)(d-2y-3)}{(d^2-2\alpha(d+2)+2y+1)^2}$$
Clearly, if $d>3$, then the minimum is reached for $\alpha=y$ and $\beta=0$. And if $d=3$ then the minimum is reached for $\alpha=0$ and $\beta=y$.
\end{proof}

Knowing these claims, we now prove that $\cut_2/\opt$ is greater than some $f_d(\alpha,\beta)$ with $\alpha+\beta>0$. To show this, we need to prove refined versions of inequalities \eqref{eq2} and \eqref{eq3}.

Recall the proof of inequality  (\ref{eq2}) ($\cut_0\geq \opt - \frac{d-1}{2}\cdot |M|$): Start with some optimum cut and remove the edges of the cut leaving $M\cap V^-$ and the edges entering  $M\cap V^+$. By definition we remove at most $\frac{d-1}{2}\cdot |M|$ edges, which implies (\ref{eq2}).

\smallskip

Let $E_0$ be the set of edges with one end in $V^-$ and the other in $M\cap V^+$, or with one end in $M\cap V^-$ and the other in $V^+$, or between two vertices of $M\cap V^+$, or between two vertices of $M\cap V^+$. We claim that 
\begin{equation}\label{eq2bis}
\cut_0 \geq  \opt- \frac{d-1}{2}\cdot |M|+|E_0|.
\end{equation}

To see this, observe that the only edges of $E_0$ that are in the optimum cut are the edges going from $M\cap V^-$ to $M\cap V^+$, and these are counted twice on our computation. The remaining edges of $E_0$ are counted once in $\frac{d-1}{2}\cdot |M|$ or $\cut_0 $ but not in $\opt$.

\medskip

We denote the stable and unstable vertices (with respect to $\cut_0$) by $S_0$ and $U_0$, respectively.


Let $E_1$ be the set of edges directed from $V^+\cap S_0$ to $V^+\cap U_0$, or from  $V^-\cap U_0$ to $V^-\cap S_0$. Observe that each such edge is added to the cut after one flip and thus
\begin{equation}\label{eq2ter}
\cut_1 \geq  \opt- \frac{d-1}{2}\cdot |M|+|E_0|+|E_1|.
\end{equation}

Let $U_1$ be the set of vertices that are unstable after 1 flip. Consider by symmetry the subset $U_1^r$ of $U_1$ that lie on the right side of the cut. These vertices were in $V^+$ (i.e.\ on the left side of the cut) before the first flip. Since each vertex of $V^+$ has deficit at least 1, the sum of the deficits of the vertices of $U_1^r$ is at least $|U_1^r|$ and thus the number of edges leaving $U_1^r$ is at least $|U_1^r|$. After one flip, by definition of $U_1$, all these edges are directed toward stable vertices on the right side of the cut. It follows that after a second flip, all these edges join the cut, and thus
\begin{equation}\label{eq2q}
\cut_2 \geq  \opt- \frac{d-1}{2}\cdot |M|+|E_0|+|E_1|+|U_1|.
\end{equation}

We now focus on finding a refined version of inequality  (\ref{eq3}) (which states that $2\cdot\opt \leq D-|M|$). Let $F_0$ be the set of edges between two vertices of $V^+\setminus M$, or between two vertices of $V^-\setminus M$. Note that each edge of $F_0$ is counted in $D$ but does not appear in $\opt$, thus we obtain
\begin{equation}\label{eq3bis}
2\cdot\opt \leq D-|M|-|F_0|.
\end{equation}

If we denote by $M^*$ the set of vertices of $M$ with deficit larger than 1 in absolute value (in other words, with deficit at least 3 or at most $-3$), the two main inequalities can be slightly refined as
\begin{equation}\label{eq3ter}
2\cdot\opt \leq D-|M|-|F_0|-|M^*|.
\end{equation}
\begin{equation}\label{eq2c}
\cut_2 \geq  \opt- \frac{d-1}{2}\cdot |M|+|E_0|+|E_1|+|U_1|+|M^*|.
\end{equation}

We are now ready to prove the following.

\begin{thm}\label{thm:mediand_flip}
Assume that $d\geq 3$ is odd. Then the 2-round algorithm consisting of the oriented median algorithm followed by two flips provides a $\frac{2}{d+1/d-3/d^2+O(d^{-3})}$-approximation for the {\sc MaxDiCut} problem in $d$-regular graphs.
\end{thm}
\begin{proof}
  We use the notation defined in this section. Let $M_1$ be the set of vertices of $M\setminus M^*$ (i.e.\ the subset of vertices of $M$ of deficit 1 in absolute value) that are not incident to any edge of $E_0$ or $E_1$. In particular all the vertices of $M_1$ are unstable, and their in-degrees and out-degrees are $\tfrac{d-1}2$ or $\tfrac{d+1}2$.

  \smallskip

  Assume first that $|M_1|\ge x \cdot |M|$, with $x=\tfrac{d^2+d}{d^2+4d+1}$. Note that for every $v\in M_1$, all the neighbors of $v$ are in $V\setminus M$, on the same side of the cut as $v$ (with respect to $\cut_0$). Since $v\in M_1$ is unstable and not incident to any edge of $E_1$, all the in-neighbors of $v$ are unstable as well. It follows that $|U_0\setminus M|\ge \tfrac{d-1}{2d}|M_1|$. Let $W$ be the subset of vertices $w\in U_0\setminus M$ that have a stable out-neighbor in $M$ (if $w\in V^+$) or a stable in-neighbor in $M$ (if $w\in V^-$). Since stable vertices of $M$ are in $M\setminus M_1$, we have $|W|\le \tfrac{d-1}2(|M|-|M_1|)$, and thus
$$
\begin{array}{rcl}
  |U_0\setminus M|-|W| & \ge &\tfrac{d-1}{2d}|M_1| -\tfrac{d-1}{2}|M|+\tfrac{d-1}{2}|M_1|\\
        &   \ge& \tfrac{d^2-1}{2d}|M_1| -\tfrac{d-1}{2}|M|\\
                       &\ge & \tfrac1{2d}(x(d^2-1)-d(d-1)) |M|\\
  & = &(1-x)|M|
\end{array}$$
by definition of $x$. Note that each vertex $u\in (U_0\setminus M)\setminus W$ is unstable, so all its neighbors are on the same side as $u$. Consider by symmetry the case $u\in V^+$. By definition, $u$ has no stable out-neighbor in $M$. If $u$ has a stable in-neighbor $v$, then the edge $uv$ is in $E_1$. If $u$ has a neighbor $v \not\in M$, then the edge $uv$ is in $F_0$. If none of these cases occur, then all neighbors of $u$ are in $M$, and are unstable. It follows that $u$ is in $U_1$, the set of vertices that are still unstable after 1 flip. It follows that
\begin{equation}   |E_1|+|F_0|+|U_1|\ge \tfrac12( |U_0\setminus M|-|W|)\ge \tfrac12(1-x) |M|.
\end{equation}

          Using inequalities (\ref{eq2q}) and (\ref{eq3ter}) together with Claims~\ref{cla:formula} and \ref{cla:f_inc}, we obtain that $\cut_2/\opt\ge f_d(\alpha,\beta)$ with $\alpha,\beta\ge 0$ and $\alpha+\beta\ge \tfrac12(1-x)$.

\medskip

Assume now that $|M_1|< x \cdot |M|$. In this case, it means that at least $(1-x)\cdot |M|$ vertices $v$ are in $M^*$ (i.e.\ have deficit at least 3 in absolute value) or are incident to an edge of $E_0$ or $E_1$. It follows that
\begin{equation}
  |M^*|+|E_0|+|E_1|\ge \tfrac12(1-x)|M|,
  \end{equation} and thus $\cut_2/\opt\ge f_d(\alpha,\beta)$ with $\alpha,\beta\ge 0$ and $\alpha+\beta\ge \tfrac12(1-x)$. 

\medskip

In both cases we obtain $\cut_2/\opt \ge f_d(\alpha,\beta)$ with $\alpha,\beta\ge 0$ and $\alpha+\beta\ge \tfrac12(1-x)=\tfrac{3d+1}{2d^2+8d+2}$. It follows from Claim \ref{cla:f_inc}, that $\cut_2/\opt\ge\tfrac2{d+1/d-3/d^2+O(d^{-3})}$.
\end{proof}

Theorem~\ref{thm:mediand_flip} proves that after 2 flips, we can slightly improve on the approximation ratio of Theorem~\ref{thm:mediand}. A natural question is whether the same can be achieved after a single flip. The construction of Figure~\ref{fig:abcdf}, which is a refinement of the construction of Figure~\ref{fig:abcd}, shows that it is not the case: if we apply the oriented median algorithm and then perform a single flip, the size of the cut does not change (it remains $\tfrac2{d+1/d}\,\opt$).

\begin{figure}[htb]
\centering
\includegraphics[scale=0.8]{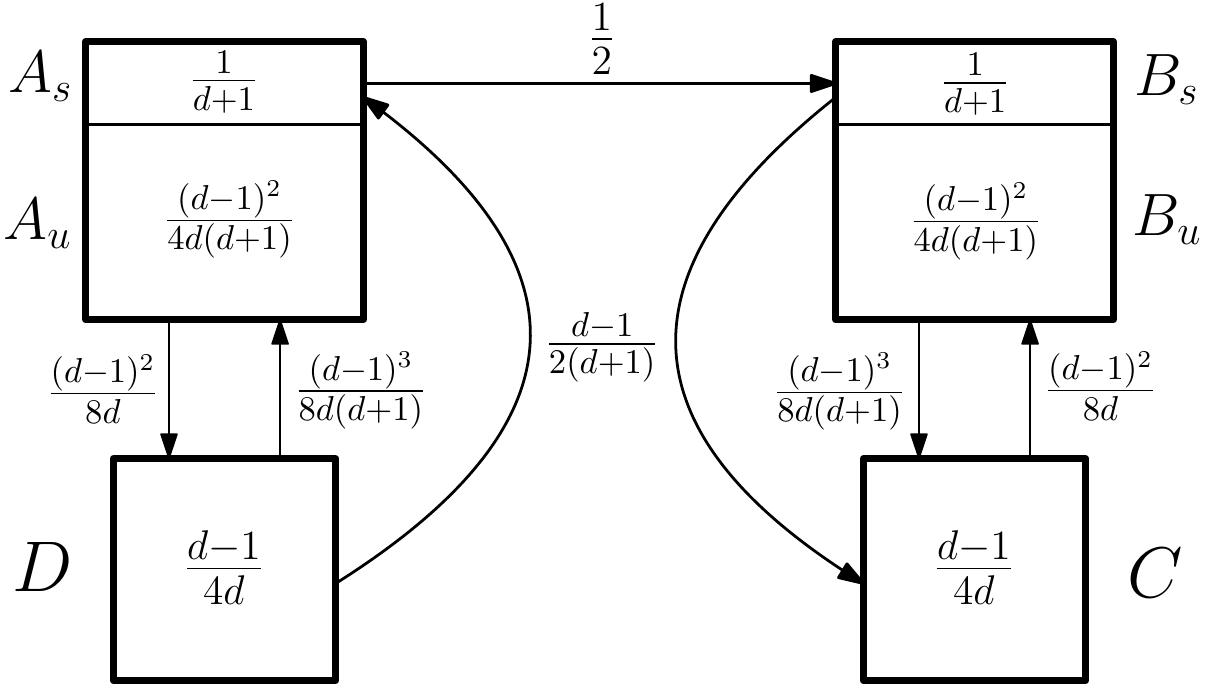}
\caption{A refinement of the construction of Figure~\ref{fig:abcd}. The subscripts $u$ and $s$ stand for \emph{unstable} and \emph{stable}, respectively.}
\label{fig:abcdf}
\end{figure} 

\section{Conclusion}

\subsection{FLIP}

In Section~\ref{sec:low}, we have designed a very simple one-round algorithm approximating {\sc MaxCut} in regular graphs (with odd degrees). Once a solution has been obtained, it might be tempting to run a few more rounds of computation to see if the solution can be improved locally.

We have already seen a simple way to improve the quality of a solution (by moving the so-called unstable vertices to the other side of the cut), but the notion of stability we used was specifically designed to improve the approximation ratio in a small number of rounds. Another simple way to locally improve a cut (in the sequential setting this time) is to take a vertex with  more neighbors in its own part than in the other part, and change its side. If this is done until no such vertex exists, the resulting cut is \emph{maximal}, and in this case is a $\tfrac12$-approximation of the maximum cut. This operation, called \emph{FLIP}, has been studied for a long time. When the edges are weighted, it was proved by Poljak~\cite{Pol95} that any sequence of FLIPs takes only polynomially many steps before reaching a maximal cut in cubic graphs, while  Monien and Tscheuschner~\cite{MT10} proved that there are graphs of maximum degree 4 for which a sequence of FLIPs can take exponentially many steps to reach a maximal cut. In the unweighted case however, since each flip improves the cut by at least one, the maximum number of flips before reaching a maximal cut is bounded by the number edges (which is linear in $n$ in bounded degree graphs). In the distributed framework, it might be tempting to consider running some rounds of the \emph{distributed FLIP} dynamics: at each round, each vertex with more neighbors in its own part than in the other part changes side. The graph $D^d_{2n}$ constructed in the previous sections shows that it might not be helpful at all: if all the vertices of the outer cycle are in one side of the cut, and all the vertices of the inner cycle are on the other side of the cut, then at each round, all the vertices of the graph would change side, not improving the solution.

It might be worth noting that in our application of the median algorithm, not all vertices of the outer cycle of $D^d_{2n}$ are on the same side of the cut (given the bad labelling of Figure~\ref{fig:ex}): due to some side-effects, roughly $d$ vertices in the outer cycle are not on the same side of the cut as the others, and similarly for the inner cycle. It can then be checked that if we run the distributed FLIP dynamics in this instance, the solution does improve over time, but improving the approximation ratio from $\tfrac1{d}$ to $\tfrac1{d}+\epsilon$ requires $\Omega(\epsilon n)$ rounds, which is extremely unpractical. This has to be compared with the lower bound of Theorem~\ref{thm:lb}, which says that in order to achieve an approximation ratio of $\tfrac1{d}+\epsilon$ in general, one needs a number of rounds of the order of $\Omega(\log^*n)$.

\subsection{Maximal cut}

An interesting aspect of the Maximal cut problem defined in the previous subsection is that it is an LCL problem (of locality 1): each vertex only needs to check that at least half of its neighbors lie on the other side of the cut. This is in stark contrast with  {\sc MaxCut}, as we have seen already. It was recently proved by Balliu,  Hirvonen, Lenzen, Olivetti, and Suomela~\cite{BHLOS19} that any deterministic algorithm finding a maximal cut in $d$-regular graphs ($d\ge 3$) takes $\Omega(\log n)$ rounds and any randomized algorithm takes $\Omega(\log \log n)$ rounds in the \textsf{LOCAL} model. It would be interesting to find algorithms matching these round complexities. Note that if we merely require that each vertex has at least $d/2-O(\sqrt{d})$ neighbors on the other side of the cut, the problem can be easily reduced to the distributed Lov\'asz Local Lemma, and therefore solved efficiently.

\begin{acknowledgement}
We would like to thank J\'er\'emie Chalopin and Keren Censor-Hillel for their remarks on the complexity of finding an orientation using very small messages in the \textsf{CONGEST} model. We also thank Michal Dory for calling reference~\cite{CHW08} to our attention, and David Gamarnik for pointing out references~\cite{CGHV15,DMS17,GL18,L17} to us.
\end{acknowledgement}


\begin{thebibliography}{99}

\bibitem{A10} M. {\AA}strand, V. Polishchuk, J. Rybicki, J. Suomela and J. Uitto. \emph{Local algorithms in (weakly) coloured graphs}, CoRR abs/1002.0125, 2010.


\bibitem{BHLOS19} A. Balliu,  J. Hirvonen, C. Lenzen, D. Olivetti, and J. Suomela, \emph{Locality of not-so-weak coloring}, In Proc. of the 26th International Colloquium on Structural Information and Communication Complexity (SIROCCO), 2019.


\bibitem{K17} K. Censor-Hillel, R. Levy and H. Shachnai, \emph{Fast Distributed Approximation for Max-Cut}. In Proc. of the 13th International Symp. on Algorithms and Experiments for Wireless Sensor Networks (ALGOSENSORS), 2017.
  

\bibitem{CKP16} Y.-J. Chang, T. Kopelowitz, and S. Pettie, \emph{An Exponential Separation between Randomized and Deterministic Complexity in the \textsf{LOCAL} Model}, In Proc. of the IEEE 57th Annual Symposium on Foundations of Computer Science (FOCS), 2016.
  
\bibitem{CGHV15} E. Cs\'oka, B. Gerencs\'er, V. Harangi, and B. Vir\'ag, \emph{Invariant gaussian processes and independent sets on regular graphs of large girth}, Random Structures Algorithms {\bf 47} (2015), 284--303.

\bibitem{CHW08}  A. Czygrinow, M. Hanckowiak, and
  W. Wawrzyniak, \emph{Fast Distributed Approximations in Planar Graphs}, In Proc. of the 22nd International Symposium on Distributed Computing (DISC), 2008.
  
   



\bibitem{DMS17} A. Dembo, A. Montanari, and S. Sen, \emph{Extremal cuts of sparse random graphs},  Ann. Probab. {\bf 45(2)} (2017), 1190--1217.

\bibitem{ER52} P. Erd\H{o}s and R. Rado, \emph{Combinatorial theorems on classifications of subsets of a given set},
Proc. London Math. Soc. {\bf 3} (1952), 417--439.

\bibitem{ES35} P. Erd\H{o}s and G. Szekeres, A combinatorial problem in geometry, Compos. Math. {\bf 2} (1935), 463--470.

  
\bibitem{GL18} D. Gamarnik and Q. Li, \emph{On the max‐cut of sparse random graphs}, Random Structure Algorithms {\bf 52(2)} (2018), 219--262.

 \bibitem{GS14} D. Gamarnik  and  M.  Sudan, \emph{Limits  of  local  algorithms  over  sparse  random graphs},   In Proc.  of  Innovations  in  Theoretical  Computer Science (ITCS), 369--376, 2014.

\bibitem{GHK17} {M. Ghaffari}, {D.G. Harris}, and {F. Kuhn}, \emph{On
    Derandomizing Local Distributed Algorithms}, In Proc. of the IEEE Symposium on Foundations of Computer Science (FOCS), 2018.

  \bibitem{GKM17} M. Ghaffari, F. Kuhn and Y. Maus, \emph{On the complexity of
      local distributed graph problems}, In Proc. of the 49th Annual ACM Symp. on Theory of Computing (STOC), 2017, 784--797.
    
\bibitem{GHS13} M. G\"o\"os, J. Hirvonen, and J. Suomela, \emph{Lower bounds for local approximation}, J. ACM {\bf 60} (2013), \#39.

 \bibitem{HLS14}   H. Hatami, L. Lov\'asz, and B. Szegedy, \emph{Limits of local-global convergent graph sequences}, Geom. Funct.
Anal. {\bf 24(1)} (2014), 269--296.

\bibitem{H17} J. Hirvonen, J. Rybicki, S. Schmid and J. Suomela, \emph{Large Cuts with Local Algorithms on Triangle-Free Graphs}, Electron. J. Combin. {\bf 24(4)} (2017), \#P4.21.

\bibitem{K12} F. Kardo\v{s}, D. Kr\'al' and J. Volec.  \emph{Maximum edge-cuts in Cubic Graphs With Large Girth
and in Random Cubic Graphs}, Random Structures Algorithms {\bf 41(4)} (2012), 506--520.

\bibitem{S17} K.-i. Kawarabayashi and G. Schwartzman, \emph{Adapting Local Sequential Algorithms to the Distributed Setting}, In Proc. 32nd International Symposium on Distributed Computing (DISC), 2018, 35:1–-35:17.



\bibitem{L92} N. Linial, \emph{Locality in distributed graph algorithms}, SIAM J. Comput. {\bf 21(1)} (1992), 193--201.

\bibitem{L17} R. Lyons, \emph{Factors of IID on trees}, Combin. Probab. Comput. {\bf 26(2)} (2017), 285--300.
  
\bibitem{MR02} M. Molloy and B. Reed, \emph{Graph Colouring and the Probabilistic Method}, Springer, 2002.

\bibitem{MT10} B. Monien and T. Tscheuschner, \emph{On the Power of Nodes of Degree Four in the Local Max-Cut Problem}, In: Calamoneri T., Diaz J. (eds) Algorithms and Complexity. CIAC 2010. Lecture Notes in Computer Science, vol 6078. Springer, Berlin, Heidelberg.

  \bibitem{Ramsey17} D. Mubayi and A. Suk, \emph{A survey of hypergraph Ramsey problems}, ArXiv e-prints, 2017.

\bibitem{N93} M. Naor and L. Stockmeyer, \emph{What can be computed locally?}, In Proc. of the 25th Annual ACM Symp. on Theory of Computing (STOC), 1993, 184--193.

\bibitem{Pol95} S. Poljak, \emph{Integer linear programs and local search for max-cut}, SIAM J. Comput. {\bf 21(3)} (1995), 450–-465.

  
\bibitem{RV17} M. Rahman and B. Vir\'ag, \emph{Local algorithms for independent sets are half-optimal},   Ann. Probab. {\bf 45(3)} (2017), 1543--1577.

\bibitem{F30} F. P. Ramsey,  \emph{On a problem of formal logic}, Proc. Lond. Math. Soc. {\bf 30} (1930), 264--286.

\bibitem{S92} J. B. Shearer, \emph{A note on bipartite subgraphs of triangle-free graphs}, Random Structures Algorithms {\bf 3(2)} (1992), 223--226.







\end{thebibliography}
\end{document}